\documentclass[a4paper,twoside,11pt,reqno]{amsart}
  \usepackage[a4paper,margin=1.2in]{geometry}  
  \usepackage[headings]{fullpage}
  \usepackage[utf8]{inputenc}
  \usepackage[usenames]{color}
  \usepackage{amssymb,eucal}
  \usepackage[all]{xy}
  \usepackage[shortlabels]{enumitem}
\usepackage[pdfpagelabels, pdftex]{hyperref}
\hypersetup{
  pdftitle={},
  pdfauthor={},
  pdfsubject={},
  pdfkeywords={},
  colorlinks=true,    
  linkcolor=blue,     
  citecolor=blue,     
  filecolor=blue,      
  urlcolor=blue,       
  breaklinks=true,
  bookmarksopen=true,
  bookmarksnumbered=true,
  pdfpagemode=UseOutlines,
  plainpages=false}
  \usepackage[nobysame,alphabetic]{amsrefs}
\newcommand{\ab}{\mathrm{ab}}
\newcommand{\id}{\mathrm{id}}
\newcommand{\bA}{\mathbf{A}} 
\newcommand{\bF}{\mathbf{F}}

\newcommand{\bP}{\mathbf{P}}
\newcommand{\bZ}{\mathbf{Z}}
\newcommand{\inj}{\hookrightarrow}
\DeclareMathOperator{\rH}{H}
\newcommand{\dO}{{\mathcal O}}

\newcommand{\et}{\text{\rm \'et}} 
\newcommand{\ep}{\varepsilon}
\newcommand{\perf}{{\rm perf}}
\newcommand{\pr}{{\rm pr}}
\newcommand{\sh}{{\rm sh}}
\newcommand{\ph}{\varphi}

\DeclareMathOperator{\Gal}{Gal}
\DeclareMathOperator{\Spec}{Spec}
\DeclareMathOperator{\Hom}{Hom}
\theoremstyle{plain}
  \newtheorem{thm}{Theorem}[section]
  \newtheorem{lemma}[thm]{Lemma}

  \newtheorem{prop}[thm]{Proposition}

\theoremstyle{definition}
  
  \newtheorem{defi}[thm]{Definition}
  \newtheorem{example}[thm]{Example}
  
 \newtheorem{notation}[thm]{Notation}

  \newtheorem{rmk}[thm]{Remark}

\newtheorem{thmABC}{Theorem}

 \newtheorem{propABC}[thmABC]{Proposition}

\numberwithin{equation}{section}
\newcommand{\tref}[1]{Theorem~\ref{#1}}
\newcommand{\secref}[1]{\S\ref{#1}}
\newcommand{\cref}[1]{Corollary~\ref{#1}}

\newcommand{\eref}[1]{Example~\ref{#1}}
\newcommand{\lref}[1]{Lemma~\ref{#1}}
\newcommand{\pref}[1]{Proposition~\ref{#1}}

\author{Piotr Achinger}
\address{Piotr Achinger, Instytut Matematyczny PAN, Śniadeckich 8, 00-656 Warsaw, Poland}
\email{pachinger@impan.pl}

\author{Jakob Stix}
\address{Jakob Stix, Institut f\"ur Mathematik, Goethe–Universit\"at Frankfurt, Robert-Mayer-Stra\ss e 6--8, 60325 Frankfurt am Main, Germany}
\email{stix@math.uni-frankfurt.de}

\title[Topological rigidity of maps in positive characteristic]{Topological rigidity of maps in positive characteristic\\ and anabelian geometry}
 
\thanks{The first author (PA) was supported by the project KAPIBARA funded by the European Research Council (ERC) under the European Union's Horizon 2020 research and innovation programme (grant agreement No 802787). The second author (JS) acknowledges support by Deutsche  Forschungsgemeinschaft  (DFG) through the Collaborative Research Centre TRR 326 "Geometry and Arithmetic of Uniformized Structures", project number 444845124. Moreover, the second author thanks IMPAN for the stimulating working conditions during a visit in 2017 which started the project of this paper.}

\date{\today}

\begin{document}

\begin{abstract}
We study pairs of non-constant maps between two integral schemes of finite type over two (possibly different) fields of positive characteristic. When the target is quasi-affine, Tamagawa showed 
that the two maps are equal up to a power of Frobenius if and only if they induce the same homomorphism on their \'etale fundamental groups. We extend Tamagawa's result by adding a purely topological criterion for maps to agree up to a power of Frobenius.
\end{abstract}

\maketitle

\DeclareRobustCommand{\SkipTocEntry}[5]{}
\setcounter{tocdepth}{1} 
{\scriptsize \tableofcontents}

\section{Introduction}
\label{sec:intro}
Anabelian geometry aims to describe geometry of schemes in terms of their \'etale fundamental groups, or more generally in terms of their \'etale homotopy types. 
Unlike in characteristic $0$, where the homotopy type of maps is locally constant in families, in characteristic $p>0$ finite \'etale covers can have continuous moduli and thus potentially determine geometry that moves in families. The most famous example here is the family of finite \'etale Artin--Schreier covers $x \mapsto \wp_t(x) = x^p - tx$ of $\wp_t \colon \bA_k^1 \to \bA_k^1$ with parameter $t$. The maps $\wp_{t,\ast} \colon \pi_1^\et(\bA_k^1,\bar 0) \to \pi_1^\et(\bA_k^1, \bar 0)$ vary with $t \in k^\times$.

Tamagawa \cite{Tamagawa2002:Pi1ofCurevsPositiveCharacteristic}*{Proposition~1.24} showed a remarkable 
\'etale rigidity property for morphisms between integral varieties over a field. 
If the target is quasi-affine then the induced map on \'etale fundamental groups determines non-constant maps up to a power of Frobenius. The \'etale rigidity of maps fits well with the result by the first author who showed in \cite{Achinger} that  affine schemes of finite type  are \'etale $K(\pi,1)$ spaces in charactersitic $p>0$. 

In this note, we show a remarkably general topological rigidity property for morphisms between integral varieties
over a field.  Non-constant maps are determined up to a power of Frobenius (the identity if the varieties are of characteristic $0$) by the map they induce on the underlying topological spaces.

\smallskip

We once and for all  fix a prime $p$. Furthermore, we denote by $|X|$ the underlying set of a scheme $X$. 
The \'etale fundamental group of a connected scheme $X$ is denoted by $\pi_1(X)$ with base points omitted. The maximal abelian quotient is $\pi_1^\ab(X)$.

The goal of this note is to prove the following result.  

\begin{thmABC}
\label{thm:main}
  Let $X$ and $Y$ be connected quasi-compact and quasi-separated $\bF_p$-schemes. Let $f,g \colon Y \to X$ be two morphisms, and consider the following properties:
  \begin{enumerate}[label=(\alph*)]
    \item 
    \label{thmitem:pi1}
    The morphisms $f_*, g_*\colon \pi_1(Y)\to \pi_1(X)$ are equal up to conjugation.
  \end{enumerate}
  \begin{enumerate}[label=(\alph*\ensuremath{'})]
    \item
    \label{thmitem:pi1ab}
    The morphisms $f_*, g_*\colon \pi^\ab_1(Y)\to \pi^\ab_1(X)$ are equal.
  \end{enumerate}
  \begin{enumerate}[label=(\alph*\ensuremath{''})]
    \item 
    \label{thmitem:H1}
    The morphisms $f^*, g^*\colon \rH^1_{\et}(X, \bF_p)\to \rH^1_\et(Y, \bF_p)$ are equal.
  \end{enumerate}
  \begin{enumerate}[label=(\alph*),resume]
    \item    
    \label{thmitem:top}
    The maps $f$ and $g$ induce the same map of sets $|Y|\to |X|$.
    \item    
    \label{thmitem:Frob}
    There exist integers $a,b\geq 0$ such that $F_X^a\circ f = F_X^b\circ g$, where $F_X\colon X\to X$ is the absolute Frobenius.
  \end{enumerate}
  Then the following implications always hold: 
  \[
  \ref{thmitem:Frob} \Longrightarrow \ref{thmitem:pi1}  \Longrightarrow \ref{thmitem:pi1ab}  \Longrightarrow \ref{thmitem:H1}  \qquad \text{and} \qquad   \ref{thmitem:Frob}  \Longrightarrow \ref{thmitem:top}.
  \]
  Moreover:
  \begin{enumerate}
  \item  \label{thmitem:1}
  If  $Y$ is integral of finite type over a field and $X$ is integral of finite type and separated over a second field, and $f(Y)$ and $g(Y)$ have positive dimension, then \ref{thmitem:top} is equivalent to \ref{thmitem:Frob}.
  \item
  If, moreover in addition to the assumptions in \eqref{thmitem:1}, $X$ is quasi-affine, then all five properties are equivalent. 
  \end{enumerate}
\end{thmABC}

\begin{rmk} 
Tamagawa in \cite{Tamagawa2002:Pi1ofCurevsPositiveCharacteristic}*{Proposition~1.24} shows that  \ref{thmitem:pi1} $\iff$ \ref{thmitem:Frob} if $Y$ is irreducible of finite type over a field, and $X$ is quasi-affine and connected. The proof of the nontrivial direction  \ref{thmitem:pi1}    $\Rightarrow$ \ref{thmitem:Frob} 
 essentially uses the equality of the pullback maps 
 \[
 f^*, g^*\colon \rH^1_{\et}(X, \bF_q)\to \rH^1_\et(Y, \bF_q)
 \]
for all $p$-powers $q$. This being a consequence of  \ref{thmitem:H1}, Tamagawa's proof in fact yields the implication 
 \ref{thmitem:H1}  $\Rightarrow$   \ref{thmitem:Frob}.

Although it turns out that our proof in \secref{sec:artin-schreier}  of  \ref{thmitem:H1} $\Rightarrow$ \ref{thmitem:top}  is similar to Tamagawa's proof of \ref{thmitem:pi1} $\Rightarrow$  \ref{thmitem:Frob}, we nevertheless decided to include our argument for completeness sake.
\end{rmk}

We start with the immediate  implications:  $\ref{thmitem:pi1}$ $\Rightarrow$ $\ref{thmitem:pi1ab}$ is trivial and \ref{thmitem:pi1ab} $\Rightarrow$ \ref{thmitem:H1}  follows from the functorial isomorphism 
\[
  \Hom(\pi^\ab_1(X), \bF_p) \simeq \rH^1_{\et}(X, \bF_p).
\]
The fact that $F_X$ induces the identity on $\pi_1(X)$ and $|F_X|$ is the identity on $|X|$ proves the 
implication \ref{thmitem:Frob}  $\Rightarrow$ \ref{thmitem:pi1} and  \ref{thmitem:top}. 

\medskip

The remaining implication  \ref{thmitem:top} $\Rightarrow$ \ref{thmitem:Frob} that completes the proof of Theorem A is a special case of the following slightly more general proposition proved in \secref{sec:topol}. Here by a variety we mean a separated scheme of finite type over a field.

\begin{propABC}[see \pref{prop:top-equal}] \label{propABC:top-equal}
  Let $Y$ be an integral scheme of finite type over a field and let $X$ be an integral variety over  a field. Let $f, g \colon Y \to X$ be two maps inducing the same map $|Y| \to |X|$, with image of positive dimension. 
  Then, there exist $a,b\geq 0$ such that 
  \[
    F_X^a\circ f = F_X^b \circ g.
  \]
\end{propABC}

In the case of varieties over finite fields, \pref{propABC:top-equal} has been proven earlier by the second author \cite{Stix}*{\S 2}. The method of \cite{Stix} applies in the more general case only partially, and at one point we need to make $Y$ to be of finite type over an uncountable field in order to reduce the problem to the case of curves. 

\medskip

We also describe a proof for the implication \ref{thmitem:H1} $\Rightarrow$ \ref{thmitem:top} \secref{sec:artin-schreier} as \pref{prop:equal on cohomology} for the convenience of the reader. We again reduce to the case of curves, and use Artin--Schreier theory to conclude. This step is similar to the proof of  \cite{Tamagawa2002:Pi1ofCurevsPositiveCharacteristic}*{Proposition~1.24} by Tamagawa.

\medskip

Before we embark on the proof, we discuss examples illustrating that the assumptions integral, affine and finite type in \tref{thm:main} are actually necessary.

\begin{example}
Let $Y=X = \Spec \bF_p[u,v]/(uv)$. Let $f\colon X \to X$ be defined by $f(u,v) = (u,v^p)$, i.e., identity on one component $C=\{v=0\}$ and Frobenius on the other $D=\{u=0\}$, and let $g=\id_X$. Since restriction
\[
\rH^1(X,\bF_p) \inj \rH^1(C,\bF_p) \oplus \rH^1(D,\bF_p)
\]  
is injective, and Frobenius acts trivially on cohomology, we find that $f^* = g^*$ on $\rH^1(X,\bF_p)$. But assertion \ref{thmitem:Frob} of \tref{thm:main} does not hold.
\end{example}

\begin{example}
The case $X=Y=\bP_k^1$ over a field $k$ of characteristic $p>0$ illustrates that the assumption that $X$ is affine cannot be dropped easily. Indeed, $\bP_k^1$ has  \'etale fundamental group $\pi_1(\bP^1_k) = \pi_1(k)$ and many nontrivial separable $k$-linear endomorphisms which all induce the identity on $\pi_1(\bP^1_k)$.
\end{example}

\begin{example}
Let $X = \bA^1_k$ with $k$ algebraically closed and let $Y = \Spec(\dO_{X,0}^\sh)$ be the strict henselisation in $0 \in \bA^1_k$. Then $\pi_1(Y) =1$ and so the distinct maps $Y \to X$ by composing the standard map with a translation of $\bA^1_k$ all induce the same map $\pi_1(Y) \to \pi_1(\bA^1_k)$.
\end{example}

\section{Preliminaries on base fields} 
\label{sec:perf}
\subsection{Reminder on inverse perfection}
\label{subsec:perfection}

Recall that for an $\bF_p$-algebra $R$, the inverse perfection is the ring
\[
  R^\perf = \varprojlim_F\, R = \lim \big(\cdots \to R \xrightarrow{F}  R \xrightarrow{F}  R \xrightarrow{F} R \big)
\]
where $F\colon R \to R$ is the Frobenius. The ring $R^\perf$ is perfect (Frobenius is an isomorphism) and inverse perfection $R^\perf$ is right adjoint to the inclusion of perfect $\bF_p$-algebras.  

\begin{lemma} \label{lem:perfection for fields}
  Inverse perfection has the following effect on fields.
  \begin{enumerate}
    \item \label{lemlabel:perfection1}
      If $R$ is a reduced $\bF_p$-algebra, then $R^\perf = \bigcap_{n \geq 0} R^{p^n}$.
    \item \label{lemlabel:perfection2}
      For a field $K$ the inverse perfection $K^\perf$ is a field. 
    \item \label{lemlabel:perfection3}
      The natural map $K^\perf \to K(T)^\perf$ is an isomorphism. 
    \item \label{lemlabel:perfection4}
      If $L/K$ is a finite purely inseparable field extension, then $K^\perf \to L^\perf$ is an isomorphism.
    \item \label{lemlabel:perfection5}
      Let $L/K$ be a finitely generated field extension. Then $L^\perf/K^\perf$ is a finite separable extension. 
  \end{enumerate}
\end{lemma}

\begin{proof}
\eqref{lemlabel:perfection1} is obvious, and \eqref{lemlabel:perfection2} follows from \eqref{lemlabel:perfection1}. 

\eqref{lemlabel:perfection3} If $f \in K(T)$ is a $p^n$-th power for all $n$, then for all discrete valuations $v$ of $K(T)$ we have $v(f) \in \bZ$ is arbitary $p$-divisible, hence $v(f) = 0$. It follows that $f \in K$. If $f = g^{p^n}$, then also $g \in K$, hence $f$ is also a $p^n$-th power for all $n$ as an element of $K$. This proves the claim. 

\eqref{lemlabel:perfection4} For large enough $n$ we have $L^{p^n} \subseteq K$. Assertion \eqref{lemlabel:perfection4} follows immediately from \eqref{lemlabel:perfection1}.

\eqref{lemlabel:perfection5} By \eqref{lemlabel:perfection3}, it suffices to treat the case where $L/K$ is a finite extension. By \eqref{lemlabel:perfection4} we may even assume that $L/K$ is finite and separable.  Let $x \in L^\perf$ and, for all $n \geq 0$, let $P_n(T)$ be the minimal polynomial over $K$ of $y_n \in L$  where $(y_n)^{p^n} = x$.

Let $Q_n(T)$ be the polynomial $P_n(T)$ with coefficients raised to $p^n$-th powers. Then $Q_n(x) = (P_n(y_n))^{p^n} = 0$, and since $L^{p^n}/K^{p^n}$ is isomorphic to $L/K$ as field extension via the $n$-th power of Frobenius, the polynomial $Q_n(T)$ is the minimal polynomial of $x \in L^{p^n}$ over $K^{p^n}$. This means that for $m \geq n$ the polynomial $Q_m(T)$ divides $Q_n(T)$ in $K^{p^n}$, but since both polynomials are monic of the same degree, they are in fact equal. This means that $Q_n(T)$ is independent of $n$, and hence has coefficients in $K^\perf$.  As $Q_0(T) = P(T)$, this polynomial is separable and of degree bounded by $[L:K]$. It follows that $L^\perf$ is a separable extension of $K^\perf$ with all elements of degree bounded by $[L:K]$. This shows that $[L^\perf:K^\perf] \leq [L:K]$ is finite.
\end{proof}

\begin{notation}
For a scheme $X$ over $\bF_p$ we write
\[
  k_X := \rH^0(X,\dO_X)^\perf.
\]
The ring $k_X$ is functorial in $X$ and the canonical map $X \to \Spec(k_X)$ is the universal map to affine schemes $\Spec(R)$ with $R$ a perfect $\bF_p$-algebra.
\end{notation}

\subsection{Reminder on the field of constants}
\label{subsec:varieties}

A variety over a field $k$ is a scheme $X$ together with a separated morphism $X \to \Spec(k)$ of finite type. The field $k$ is not unique, but there is usually a universal one. 

\begin{prop}
\label{prop:universal constants}
  Let $X$ be an integral variety over a field $k$. Then there is a field $L_X \subseteq \rH^0(X,\dO_X)$ such that, for every field $L$, every map $X \to \Spec(L)$  uniquely factors over the map induced by an inclusion $L \inj L_X$. In particular, $L_X$ is the unique maximal subfield of $\rH^0(X,\dO_X)$ containing all other subfields and is functorial in $X$.

  Moreover, $L_X$ is a finite extension of $k$, and if $X$ is normal then $X$ is geometrically integral as a variety over $L_X$.
\end{prop}

\begin{proof}
The proof is essentially identical to the proof in the affine case \cite{Tamagawa1997:GrothendieckConjectureAffine}*{Lemma~4.2} by Tamagawa. We spell out the details for the convenience of the reader: $L_X$ is simply the integral closure of $k$ inside $R = \rH^0(X,\dO_X)$. We claim that this is already the desired field $L_X$. If there is a subfield $L \subseteq R$ not contained in $L_X$, then there is an element $t \in L \setminus L_X$ which is necessarily transcendental over $L_X$ and hence over $k$. Let $\bF$ be the prime field of $k$, so that $\bF(t) \subseteq L \subseteq R$. The inclusion $\bF[t] \subseteq k[t] \subseteq R$ induce maps 
\[
  X \to \Spec (k[t]) \to \Spec (\bF[t]).
\]
The first map is a dominant $k$-linear map of varieties over $k$, hence by Chevalley's Theorem has constructible image, and due to dimension $1$ has in fact an open image. The composition with the second map projects onto the generic point. This is absurd.

For finiteness of $L_X$ over $k$, notice that $L_X$ is contained in the field of rational functions $k(X)$, which is finitely generated over $k$, so $L_X$ is finitely generated over $k$. It is also algebraic over $k$ by construction, and therefore finite over $k$. 

For the final assertion, if $X$ is normal, then $L_X$ concides with the algebraic closure of $k$ in $k(X)$. So $L_X$ is algebraically closed in $k(X)$, and the result follows from \cite{EGAIV2}*{Corollaire~4.6.3}.
\end{proof}

If $X$ is not normal, then it might not be geometrically irreducible over $L_X$. A concrete example is as follows. 

\begin{example} \label{ex:base change fails integral}
Assuming characteristic not $2$, we may take a separable quadratic extension $K = k(\sqrt{a})$ with nontrivial $\sigma \in \Gal(K/k)$ and glue $0 \in \bA^1_K$ to itself via $\sigma$. The resulting scheme is $X = \Spec(A)$ with 
\[
  A = \{f(t) \in K[t] \ ; \ f(0) \in k\} = k[t,y]/(at^2 -  y^2)
\]
where $y = \sqrt{a} t$. This $X$ is integral, but $X_K$ is not, since 
\[
  A\otimes_k K = K[t,y]/(y^2 - at^2) = K[u,v]/(uv)
  \quad 
  \text{with $u = y - \sqrt{a} t$ and $v = y + \sqrt{a} t$.}
\] 
Note that if $U = \Spec(A[\frac 1 t])$ is the complement of the singular point, then $L_U = K \neq L_X$.
\end{example}

\subsection{Inverse perfection as a field of constants}

If the base field is perfect, then we can use the inverse perfection to detect the maximal field of constants. 

\begin{defi}
  A \textbf{variety over a perfect field} is a scheme $X$ such that there is a perfect field $k$ and a separated map of finite type $X \to \Spec(k)$. 
\end{defi}

\begin{lemma}
\label{lem:kX as LX}
  Let $X$ be an integral scheme over $\bF_p$. Then the following are equivalent.
  \begin{enumerate}[label=(\alph*)]
    \item \label{lemitem:varperfect1}
      $X$ is a variety over a perfect field.
    \item \label{lemitem:varperfect2}
      The ring $k_X$ is a perfect field and the map $X \to \Spec(k_X)$ is separated of finite type.
  \end{enumerate}
  If both properties hold, then $k_X$ is the unique maximal subfield $L_X$ of $\rH^0(X,\dO_X)$. 
\end{lemma}

\begin{proof}
Assertion \ref{lemitem:varperfect2} clearly implies  \ref{lemitem:varperfect1}. So we prove the converse and assume that we have a perfect field $k$ and a separated map of finite type $X \to \Spec(k)$. 

Denote by $k(X)$ the function field of $X$. The inclusions $k \inj \rH^0(X,\dO_X) \inj k(X)$ yield inclusions
\[
  k = k^\perf \inj  k_X \inj k(X)^\perf.
\]
By \lref{lem:perfection for fields} \eqref{lemlabel:perfection5} applied to $k(X)/k$ the extension $k_X/k$ is contained in a finite (separable) extension. So $k_X$ is perfect and  \ref{lemitem:varperfect2} holds.

We now assume that both properties hold and keep the notation above. \pref{prop:universal constants} shows that $L_X$ is a finite extension of $k$, and hence is perfect. In particular, $L_X$ is contained in $\rH^0(X,\dO_X)^\perf = k_X$. The converse inclusion is obvious. 
\end{proof}

\begin{rmk}
Let $X$ be a variety over a perfect field $k$. Any open subscheme $U \subseteq X$ is a variety over a perfect field. Note that the field extension $k_X \to k_U$ can be nontrivial if $X$ is not normal, see Example~\ref{ex:base change fails integral}.
\end{rmk}

\section{Topological coincidence of maps, revisited}
\label{sec:topol}

The goal of this section is to prove the following rigidity property of maps between integral varieties. In case the base fields have characteristic $0$, the Frobenius maps have to be interpreted as the identity, and some steps in the proof can be left out. We mainly deal with the case of base fields of characteristic $p$. 

\begin{prop} \label{prop:top-equal}
  Let $Y$ be an integral scheme of finite type over a field, and let $X$ be an integral variety over a field. Let $f, g \colon Y \to X$ be two maps inducing the same map $|Y| \to |X|$, with image of positive dimension. 

  Then, there exist $a,b\geq 0$ such that $F_X^a\circ f = F_X^b \circ g$.
\end{prop}

Take note that we do not assume that $X$ and $Y$ are defined over the same field. If both $X$ and $Y$ are of finite type over $\bF_p$, the proposition follows from  \cite{Stix}*{Proposition~2.3}. Since we did not succeed to reduce to this case, we give a completely independent proof in the complementary case that instead exploits passing to uncountable fields. Surprisingly, this second proof strategy below needs the additional assumption that {\bf $L_X$ is infinite}, hence complementing  \cite{Stix}*{Proposition~2.3}. Below we give a complete proof in both cases. 

\subsection{Reduction to functions}
\label{subsec:topol reduction to functions}

Before embarking on the technical heart of the proof of \pref{prop:top-equal}, we will perform a few easy reductions. 

\begin{lemma} \label{lem:first-reductions}
  In the situation described in \pref{prop:top-equal}.
  \begin{enumerate}[(1)]
    \item \label{lemitem:compose with dominant}
      Let $h\colon Y'\to Y$ be a dominant map with $Y'$ an integral scheme of finite type over some field, and suppose that $fh$ (and equivalently $gh$) still has image of positive dimension. 
      
      If the assertion of \pref{prop:top-equal} holds for $fh,gh\colon Y'\to X$, then it holds for $f,g\colon Y\to X$ as well.
    \item \label{lemitem:compose with immersion}
      Let $i\colon X\to X'$ be a locally closed immersion with $X'$ an integral variety. If the assertion of \pref{prop:top-equal} holds for $if,ig\colon Y\to X'$, then it holds for $f,g\colon Y\to X$ as well.
    \item \label{lemitem:target affine} 
      Let $U \subseteq X$ be a dense open subvariety such that the intersection of $U$ with the image of $f$ (and equivalently for $g$) still has positive dimension. Denote the preimage $f^{-1}(U) = g^{-1}(U)$ by $V$ and the restrictions of $f$ and $g$ by $f_U,g_U : V \to U$. If the assertion of \pref{prop:top-equal} holds for $f_U,g_U\colon V\to U$, then it holds for $f,g\colon Y\to X$ as well.
  \end{enumerate}
\end{lemma}

\begin{proof}
Part \ref{lemitem:compose with dominant} follows from the fact that, since $X$ is separated, the equality of $f'h=g'h$ implies that $f'$ agrees with $g'$ for any $f', g'\colon Y\to X$. 
Part \ref{lemitem:compose with immersion} similarly follows from the fact that $i$ is a monomorphism. For \ref{lemitem:target affine}, let $j \colon U\to X$ and $h \colon V\to Y$ be the inclusions. The assertion for $(f_U, g_U)\colon V\to U$ implies the assertion for $(jf_U, jg_U)=(fh, gh)\colon V\to X$, which implies the assertion for $(f, g)\colon Y\to X$ by \ref{lemitem:compose with dominant}.
\end{proof}

\begin{prop} \label{prop: target A1}
  In Proposition~\ref{prop:top-equal}, we may assume that $Y$ is an integral normal variety over an uncountable algebraically closed field $K=k_Y$ and $X = \bA^1_k$ with $k$ a field. 
\end{prop}

\begin{proof} 
The proof proceeds in three steps. We first pick an uncountable algebraically closed field $K$ that is an extension of $L_Y$ and consider the normalization $Y_0'$ of an irreducible component $Y_0$ of the base change $Y_K = Y \otimes_{L_Y} K$. (See  \eref{ex:base change fails integral} for a case  where  $Y_K$ itself is not integral.) Note that $Y'_0$ is also integral and of finite type over the perfect field $K$.
The composition with the projection 
\[
  h\colon Y'_0 \to Y_0 \inj Y_K = Y \otimes_{L_Y} K \to Y
\]
is schematically dominant. By choosing $Y_0$ appropriately, we may also assume that for the maps $f$ and $g$ for which we are trying to prove the claim of Proposition~\ref{prop:top-equal}, the compositions $fh$ and $gh$ have still an image of positive dimension. Now apply \lref{lem:first-reductions}\ref{lemitem:compose with dominant}. 

In the second step we consider an affine open $U$ in $X$ that contains two points of $f(Y) = g(Y)$ that show this image has dimension $>0$. By \lref{lem:first-reductions}\ref{lemitem:target affine} we may replace $X$ by $U$ (and $Y$ by $f^{-1}(U) = g^{-1}(U)$; this keeps the properties of $Y$ we already adjusted) and thus assume $X$ is an affine variety over a field $k$. We now choose a closed embedding $\iota: X \inj \bA^n_k$, and \lref{lem:first-reductions}\ref{lemitem:compose with immersion} lets us assume $X = \bA^n_k$. 

In the third step we choose good coordinates on $\bA^n_k$ so that we may reduce to the coordinate projections.\footnote{The reduction to coordinate projections is obvious in characteristic $0$.} We set $B = \rH^0(Y,\dO_Y)$ and have $k_Y = \bigcap_{n \geq 0} B^{p^n} = K$ by \lref{lem:kX as LX}. 
Let $x_1, \ldots, x_n$ be the linear standard coordinates on $\bA^n_k$, and we set $f_i = f^\ast(x_i)$ and $g_i = g^\ast(x_i)$. Then not all $f_i$ are contained in $K$ because otherwise the map $f: Y \to \bA^n_k$ would factor over  a point $\Spec(K)$ and not have image of positive dimension. After replacing $f$ by a map $f' : Y \to \bA^n_k$ with $f = F_{\bA^n_k}^a \circ f'$ we may assume that there is an index $j$ with $f_j \in B \setminus B^p$. 
Indeed, if all $f_i\in B^p$, then the map $f^\ast \colon k[x_1, \ldots, x_n] \to B$ has image contained in $B^p$ (as $k$ maps into $K=K^p$) and so $f$ factors as $f = f' \circ F_Y$. But then also $f = F_{\bA^n_k} \circ f'$.

As $B^p \subseteq B$ is a proper subgroup, we may replace (by a linear coordinate change $x_i \mapsto x_i + x_j$ of coordinates of $\bA^n_k$) any $f_i \in B^p$ by $f_i + f_j$ to assume that for all $i = 1, \ldots, n$ we have $f_i \in B \setminus B^p$. 

Now we turn our attention to $g$. By factoring over a suitable power of Frobenius, again for one index $j$ we have $g_j \in B \setminus B^p$. Since we prepared the coordinates of $\bA^n_k$ carefully with respect to $f$, we are guaranteed that with the same index $j$ we have $f_j, g_j \in B \setminus B^p$.  Note that $f_j(Y) = g_j(Y)$ in $\bA^1_k$ does not factor over a closed point because otherwise $f_j$ and $g_j$ would be $p$-th powers. 

We are now modifying the coordinates on $\bA^n_k$ yet again by a transformation of the form $x_i \mapsto x_i + h_i(x_j)$ for all $i \not= j$. We pick polynomials $h_i \in k[T]$ such that  $f_i + h_i(f_j)$ and $g_i + h_i(g_j)$ are not contained in $B^p$. The polynomials $h_i$ that are unsuitable here are contained in two proper affine linear subspaces.  More precisely, we consider the maps of $k$-vector spaces
\begin{align*}
  \ph \colon k[T] \to B/B^p, \qquad h(T) \mapsto h(f_j) + B^p, \\
  \psi \colon k[T] \to B/B^p, \qquad h(T) \mapsto h(g_j) + B^p,
\end{align*}
and need to show that there is an $h_i$ such that $\ph(h_i) \notin -f_i + B^p$ and $\psi(h_i) \notin -g_i + B^p$. The subgroup  
\[
  \ker(\ph) = \{h \in k[T]  \ ; \   h(f_j) \in B^p\} 
\]
is in fact a normal subring of $k[T]$, because $B$ is normal. It contains $k[T^p]$ but not $T$, because $f_j \notin B^p$. It follows that $\ker(\ph) = k[T^p]$, and similarly $\ker(\psi) = k[T^p]$. Since we established that $\ker(\ph) = \ker(\psi)$ has codimension $\infty$ in $k[T]$, we always find a polynomial $h_i$ avoiding the two forbidden affine subspaces.\footnote{Unless $k=\bF_2$, the precise form of $\ker(\ph)$ and $\ker(\psi)$ is not important, only that these are proper subspaces.}

We have now achieved that the maps $f,g : Y \to \bA^n_k$ are coordinatewise not $p$-th powers. In particular, the coordinate functions $f_i, g_i$ are not constant since otherwise $k(f_i)$ or $k(g_i)$ are contained in $K$, the maximal subfield of $B = \rH^0(Y,\dO_Y)$, and thus $p$-th powers contrary to our preparations. Since $|f| = |g|$ also implies $|f_i| = |g_i|$, we may apply now  Proposition~\ref{prop:top-equal} to these coordinate functions and deduce that there are $a_i,b_i \geq 0$ with 
\[
  F^{a_i}_{\bA^1_k} \circ f_i =  F^{b_i}_{\bA^1_k} \circ g_i.
\]
After canceling powers of Frobenius we may assume that $\min\{a_i,b_i\} = 0$ for each $i$. But since neither $f_i$ nor $g_i$ is a $p$-th power, we obtain $a_i = b_i = 0$ for all $i$. It follows that $f = g$.
\end{proof}

\subsection{Reduction to smooth curves} 
\label{subsec:topol reduction to curves}

We are going to reduce further to ``generically \'etale'' maps from a connected smooth affine curve to the affine line.

\begin{lemma} \label{lem:dim Y  equals 1}
  In Proposition~\ref{prop:top-equal}, in addition to the reduction of 
  \pref{prop: target A1},  we may assume that $Y$ is a connected smooth affine curve.
\end{lemma}

\begin{proof}
We need to argue that Proposition~\ref{prop:top-equal} holds in the scenario of \pref{prop: target A1}. So let $Y$ be an integral variety over an uncountable algebraically closed field $K$, and let $f,g : Y \to \bA^1_k$ be maps with $|f| = |g|$ having image of dimension $>0$. For $a,b \geq 0$, let 
\[
  Z_{a,b}\subseteq Y
\]
be the equalizer of $F_X^a\circ f$ and $F_X^b\circ g$, a closed subscheme of $Y$. We need to show that there are suitable $a,b$ with $Z_{a,b} = Y$.

We first note that, $f,g$ having image of dimension $>0$, implies that the generic point and at least one closed point  of $\bA^1_k$ is in the image. In particular, the  induced map of $f$ and $g$ on closed points $Y(K) \subseteq |Y| \to |\bA^1_k|$ is non-constant. 

For every $y\in Y(K)$, pick an integral curve $C_y \subseteq Y$ passing through $y$ and another point $z$, depending on $y$, with $f(y) \not= f(z)$, see \cite{mumford:AV}*{Lemma, p.\ 56}. Let $D_y$ be an affine open subset of the normalization of $C_y$ such that the image of the induced map $\gamma\colon D_y\to Y$ contains $y$ and $z$. Thus $\gamma f$ and $\gamma g$ are not constant, and by the assumed case of the result applied to $f\gamma, g\gamma\colon D_y\to X$, we see that $C_y(K) \subseteq Z_{a,b}(K)$ for some $a,b \geq 0$. In particular, $y\in Z_{a,b}$. Since this holds for all $y \in Y(K)$, by the \lref{lem:countable union of constructible} below there is a nonempty open $U \subseteq Y$ contained in some $Z_{a,b}$. Since $Z_{a,b}$ is closed and $Y$ is reduced, we have $Y = Z_{a,b}$ and Proposition~\ref{prop:top-equal} holds for $f,g: Y \to \bA^1_k$. 
\end{proof}

\begin{lemma} \label{lem:countable union of constructible}
  Let $Y$ be an integral scheme of finite type over an uncountable algebraically closed field $K$ and let $\{Z_i\}_{i\in I}$ be a countable family of constructible subsets of $Y$ such that
  \[ 
    Y(K) = \bigcup_{i\in I} Z_i(K).
  \] 
  Then there exists an $i\in I$ such that $Z_i$ contains a dense open subset of $Y$. 
\end{lemma}

\begin{proof}
This fact is well-known, though we do not know of a reference; here is a sketch of the proof. We may first replace $Y$ by a dense affine open. Then Noether normalization allows us to reduce to the case $Y\simeq \bA^n_K$. We argue by induction on $n$, the case $n=0$ being obvious. Suppose that $n>0$ and that none of the $Z_i$ is dense. Then each of their closures contains at most a finite number of hyperplanes. By uncountability of $K$, there exists a hyperplane $\bA^{n-1}_K\simeq H\subseteq \bA^n_K$ not contained in the closure of any $Z_i$, contradicting the induction assumption for the intersections $H \cap Z_i$ whose $K$-points cover $H(K)$.
\end{proof}

In the curve case, we can factor out powers of Frobenius to assume that $f$ and $g$ are both ``generically \'etale.'' To make this precise, we need to introduce some notation. Consider the situation of \lref{lem:dim Y  equals 1}: $Y=\Spec(B)$ is a connected smooth affine curve over the big field $K$ and we have two nonconstant maps $f, g\colon Y\to X=\bA^1_k$. Let $f_0, g_0\colon k\to K$ be the induced field extensions, see the functoriality of the field of constants in \pref{prop:universal constants}. Focusing on $f$, it corresponds to a map $k[x]\to B$ inducing $f_0$ on $k$ and mapping $x\mapsto f'$. Geometrically, the map $f'\colon Y\to \bA^1_K$ appears in the factorization
\begin{equation} \label{eq:factorization}
  \xymatrix@M+1ex{
    Y \ar[r]^-{f'} \ar[dr]_{f} & X_K = \bA^1_K \ar[d]^{f_0} \\
    & X = \bA^1_k
  }
\end{equation}
where by abuse of notation we also denote the projection $X_K\to X$ of the base change along $f_0 : k \to K$ by $f_0$. Thus $f'$ is not constant, and since $B^\perf=K$ there exists an integer $a\geq 1$ and an element $f'_{\rm new}\in B\setminus B^p$ such that $f' = (f'_{\rm new})^{p^a}$. Since $Y$ is a connected smooth curve, the fact that $f'_{\rm new}$ is not a $p$-th power means that $df'_{\rm new}\neq 0$, so that the induced map $f'_{\rm new} \colon Y\to \bA^1_K$ is generically \'etale. We write $f_{\rm new}\colon Y\to \bA^1_k$ for the  composition with $f_0\colon \bA^1_K \to \bA^1_k$, and find $f = F_{ \bA^1_k}^a\circ f_{\rm new}$. We may therefore for our purposes replace $f$ with $f_{\rm new}$ and hence assume that in addition $f'\colon Y\to\bA^1_K$ is generically \'etale. Doing the same with $g$, we conclude that Proposition~\ref{prop:top-equal} follows from the assertion below.  

\begin{prop}
\label{prop:actually equal if top-equal to A1}
Let $Y$ be a connected affine smooth curve over $k_Y = K$ and let $f, g \colon Y \to \bA^1_k$ be two maps inducing the same map $|Y| \to | \bA^1_k|$ such that the induced maps $f',g'\colon Y\to \bA^1_K$ are generically \'etale. 

  Then, there exist $a,b\geq 0$ such that $F_{\bA^1_k}^a\circ f = F_{\bA^1_k}^b \circ g$.
\end{prop}

We are going to prove \pref{prop:actually equal if top-equal to A1} in the following two sections depending on the cardinality $\#k$. 

\begin{rmk}
In fact, in view of our reduction to the generically \'etale case, the conclusion of 
\pref{prop:actually equal if top-equal to A1} should be $f$ equals $g$. This is indeed what we prove below when $\#k$ is infinite. However, when $\#k$ is finite, we again trade in \textit{generically \'etale} for another desirable property, namely that $f$ and $g$ agree on the field of constants. We decided to weaken the assertion of 
\pref{prop:actually equal if top-equal to A1} in favour of a more transparent structure of our reduction steps in the proof. 
\end{rmk}

\subsection{Infinite base field} 
\label{subsec:topol infinite base field}

We first assume that $k$ is an infinite field. By \lref{lem:first-reductions}\ref{lemitem:compose with dominant}, we may shrink $Y$ further to an open subscheme so that both maps $f', g' : Y \to \bA^1_K$ are \'etale. 

Let $U \subseteq \bA^1_k$ be the image of $f$ and equally $g$. We verify that $U$ is an open subset of $\bA^1_k$: factoring $f$ into $f'\colon Y\to\bA^1_K$ and $f_0\colon \bA^1_K\to\bA^1_k$, the set $U=f_0(f'(Y))$ is the image of $f'(Y)\subseteq\bA^1_K$, which is the complement of a finite subset. Then $f_0^{-1}(U)$ and $g_0^{-1}(U)$ are open subschemes of $\bA^1_K$ which receive the maps $f'$ and $g'$. 

We claim that for $y_1,y_2 \in Y(K)$ we have $f'(y_1)=f'(y_2)$ if and only if $g'(y_1)=g'(y_2)$. In other words, the \'etale equivalence relation 
\[
  R(f') = Y \times_{f',\bA^1_K, f'} Y \subseteq Y \times_K  Y
\]
agrees with the analogous equivalence relation $R(g')$. Note that  $Y/R(f') \simeq f'(Y)$ canonically induced by $f'$, and similarly for $g'$.

Indeed, it suffices to check this for $y_1$ and $y_2$ belonging to an infinite subset $Z \subseteq Y(K)$. Consider $Z$ to be the set of all $y \in Y(K)$ whose image in $\bA^1_k$ (via $f$ or $g$) is a $k$-rational point. The set $Z$ is infinite by our assumption that $\bA^1(k) = k$ is infinite(!). Now, the preimage of any $k$-rational point under the projections $f_0, g_0: \bA^1_K \to \bA^1_k$ is a singleton. This shows that for $y_1, y_2 \in Z$ we have
\[ 
  f'(y_1)=f'(y_2)   \iff  f(y_1)=f(y_2)  \iff g(y_1)=g(y_2) \iff   g'(y_1)=g'(y_2),
\]  
where the equivalence in the middle makes use of $|f| = |g|$.

Therefore $f'$ and $g'$ define the same \'etale equivalence relation on $Y$, which implies that there exists a $K$-linear isomorphism 
\[
  \ph \colon  f'(Y) \xleftarrow{f'} Y/R(f') = Y/R(g') \xrightarrow{g'} g'(Y).
\]
Then, we have the diagram
\[ 
  \xymatrix{
    & f'(Y) \ar[dr]^{f_0} \ar[dd]_\simeq^\varphi  \\
    Y\ar@{->>}[ur]^{f'} \ar@{->>}[dr]_{g'} & & U \\
    & g'(Y) \ar[ur]_{g_0} 
  }
\]
where the left triangle commutes by construction and the right triangle commutes on the level of topological spaces as a consequence of $|f| = |g|$.

Let $\overline \varphi \colon \bP^1_{K}\to \bP^1_{K}$ be the unique extension of $\varphi$ to smooth compactifications. For $\alpha \in U(k)$ the preimage under $f_0$ (resp.\ $g_0$) is a singleton, more precisely only the point $[f_0(\alpha):1]$ (resp.\ the point $[g_0(\alpha):1]$), hence 
we have 
\begin{equation} \label{eq:constraint on ph}
  \overline \varphi ([f_0(\alpha):1])= [g_0(\alpha):1], \qquad \text{ for } \alpha \in U(k).
\end{equation}
By a change of coordinates defined over $k$ we may assume that $0$ and $1$ belong to $U$. Then \eqref{eq:constraint on ph} yields in particular that $\overline \ph$ fixes $0 = [0:1]$ and $1 = [1:1]$.  After a further change of coordinates by $t \mapsto t/(t-1)$, now $0$ and $\infty = [1:0]$ are preserved. In this coordinate the map $\overline{\varphi}$ must be of the form, for some $\lambda \in K$, 
\[
  \overline{\varphi}(t) = \lambda \cdot t, \qquad t \in K \cup \{\infty\}.
\]
Now \eqref{eq:constraint on ph} implies 
\[
  \lambda = g_0(\alpha)/f_0(\alpha) \qquad \text{ for all } \alpha \in U(k), \alpha \not= 0.
\]
As $U(k)$ contains all but finitely many elements of $k^\times$ and $k$ is an infinite field, we find elements
$\alpha_1,\alpha_2 \in U(k)$ with $\alpha_1 \alpha_2 \in U(k)$. Then we conclude $\lambda = 1$ from 
\[
  \lambda^2 = \frac{g_0(\alpha_1)}{f_0(\alpha_1)} \cdot \frac{g_0(\alpha_2)}{f_0(\alpha_2)} 
  = \frac{g_0(\alpha_1\alpha_2)}{f_0(\alpha_1\alpha_2)} = \lambda.
\]
Thus $\overline \varphi$ is the identity $g' = \ph \circ f'$ equals $f'$. Moreover, also $f_0 = g_0$ holds for the cofinite set $U(k) \cap k^\times$, and thus for all of $k$. This completes the proof of \pref{prop:actually equal if top-equal to A1} if $\#k$ is infinite.  

\subsection{Finite base field: intersection theory} 
\label{subsec:topol finite base field}

We will now deal with the case $X = \bA_k^1$ with $k=k_X$ a finite field of cardinality $q$. We use $f_0$ to identify $k$ with a subfield of $K = k_Y$. Then, as $K$ only contains a unique subfield of cardinality $q$, we have $g_0(k) = k$, and so there is some $b \geq 0$ such that $f_0 = g_0 \circ F^b$. We may replace $g$ by the composition $F^b \circ g$, and thus assume that $f_0 = g_0$ at the expense of giving up that $g'$ is actually generically \'etale ($f'$ still is). We thus have achieved that in the factorization of $f$ and $g$ as in \eqref{eq:factorization} both $K$-varieties denoted $X_K$ and the projections $X_K \to X$ agree. Since we will be mostly working with the $K$-linear maps now, we rename $\pr := f_0 = g_0$ and denote $f'$ by $f$ (resp.\ $g'$ by $g$). So we are left to prove the following.

\begin{prop} \label{prop:top-equal to A1 finite field target}
Let $k$ be a finite field contained in an algebraically closed field $K$. Let $Y$ be a connected affine smooth curve of finite type over $K$ and let $f, g \colon Y \to \bA^1_K$ be two maps of $K$-schemes inducing the same map $|Y| \to |\bA^1_k|$ after composition with the projection $\pr \colon \bA^1_K \to \bA^1_k=X$, and with image in $|\bA^1_k|$  of positive dimension. We further assume that $f$ is generically \'etale.
  
  Then, there exist $a,b\geq 0$ such that $F_X^a\circ f = F_X^b \circ g$.
\end{prop}

\begin{proof}
Let $\eta_Y$ (resp.\ $\eta_{X_K}$ and $\eta_X$) denote the generic point of $Y$ (resp.\ $X_K$ and $X$). Let $\bar k$ denote the algebraic closure of $k$ inside $K$. Then the maps of topological spaces to be considered factor as
\[
|Y| =   \{\eta_Y\} \cup Y(K)  \xrightarrow{|f|, |g|}  |X_K| =  \{\eta_{X_K}\} \cup K \xrightarrow{|\pr|} |X| = \{\eta_{X}\} \cup \bar k/\Gal(\bar k/k)
\]
where by $\bar k /\Gal(\bar k/k)$ we denote the set of Galois orbits of $\Gal(\bar k/k)$ acting on $\bar k$. The fibre of $\pr$  in a Galois orbit agrees with precisely this orbit. 

For every $d \geq 1$ there is a unique subfield $k_d \subseteq K$ of order $q^d$. The set $G_d := \pr(k_d)$ in $|X|$ consists of precisely those Galois orbits of length dividing $d$.  Since $k_d = \pr^{-1}(G_d)$ and since $|\pr \circ f|$ agrees with $ |\pr \circ g|$,  the preimages with respect to $f,g$
\[
  f^{-1}(k_d) = |f \circ \pr|^{-1}(G_d)  = |g \circ \pr|^{-1}(G_d)  = g^{-1}(k_d) 
\]
agree and will be denoted by $S_d$. Let 
\[
  M_d =  \left\{ m \in \bZ  \ ; \ -\frac{d}{2} < m \leq \frac{d}{2}\right\}
\]
be a minimal set of representatives of $\bZ/d\bZ$ in terms of absolute value. It follows that for every $y \in S_d$ there is an $m \in M_d$ with 
\begin{equation} \label{eq:pointwise Frob relation}
  \begin{cases}
    \ \ \quad f(y) = g(y)^{q^m} & \text{ if $m \geq 0$} ,\\
    f(y)^{q^{-m}} = g(y) \ \quad & \text{ if $m \leq 0$} ,\\
  \end{cases}
\end{equation}
because $\Gal(k_d/k)$ is generated by the $q$-power Frobenius. 

For $m \geq 0$, let $\Gamma_m \subseteq \bP^1_K \times \bP^1_K$ denote the graph of Frobenius $F^m : \bP^1_K \to \bP^1_K$, given by $[u:v] \mapsto [u^{q^m}: v^{q^m}]$. For $m<0$, we denote by $\Gamma_m$ the image of $\Gamma_{-m}$ under the transposition of factors of $\bP^1_K \times \bP^1_K$.  

Let $\overline Y$ denote the smooth projective completion of $Y$, and let $\bar f$ and $\bar g$ denote the extensions of $f$ and $g$ to maps $\bar f, \bar g \colon \overline Y \to \bP^1_K$.  Define $\bar Z_m$ as the fibre product
\[
  \xymatrix@M+1ex@R-1ex@C-1ex{
    \overline Z_m \ar[d] \ar@{^(->}[r] & \overline Y \ar[d]^{h = (\bar f, \bar g)} \\
    \Gamma_m \ar@{^(->}[r] & \bP^1_K \times \bP^1_K. }
\]
If $\overline Z_m = \overline Y$, then $f = F^m \circ g$ for $m \geq 0$ or  $F^{-m} \circ f = g$ for $m<0$, so \pref{prop:top-equal to A1 finite field target} is proven. We argue by contradiction and assume that for all $m \in \bZ$ the subscheme $\overline Z_m$ has dimension $0$.  Its degree can be computed as
\[
  \deg(\overline Z_m) = \begin{cases}
    \deg(f) + q^m \deg(g) & \text{ if $m \geq 0$} ,\\
    q^{-m} \deg(f) + \deg(g) & \text{ if $m \leq 0$} .\\
  \end{cases}
\]
By \eqref{eq:pointwise Frob relation} we have
\[
  S_d \subseteq  \bigcup_{m \in M_d} \bar{Z}_m(K).
\]
Now, since $f$ is generically \'etale, there is a fixed $B$, uniform in $d$, taking into account the ramification and boundary of $f\colon Y \to \bA^1_K$ in comparison with $\bar f\colon \overline Y \to \bP^1_K$, such that 
\[
  \deg(f) \cdot q^d - B \leq \#S_d.
\]
Combining the above we obtain the inequality
\[
  \deg(f) \cdot q^d - B \leq \#S_d \leq \sum_{m \in M_d} \deg \overline Z_m
  \leq d \cdot q^{d/2} \big(\deg(f) + \deg(g)\big) .
\]
Letting $d$ tend to infinity leads to a contradiction, and that concludes the proof of \pref{prop:top-equal to A1 finite field target} and thus also, finally, the proof of \pref{prop:top-equal}.
\end{proof}

\section{Artin--Schreier theory}
\label{sec:artin-schreier}

The goal of this section is to prove the implication \ref{thmitem:H1}$\Rightarrow$\ref{thmitem:top} of Theorem~\ref{thm:main}, thus completing its proof. We first deal with the case of a ``formal punctured disc''. 

\subsection{Local curve case}

Temporarily, let $K$ be a perfect field of characteristic $p>0$. The field $K(\!(t)\!)$ of formal Laurent series is endowed with the discrete valuation $v\colon K(\!(t)\!)\to \mathbf{Z}\cup\{\infty\}$ normalized so that $v(t)=1$. We write $d\colon K(\!(t)\!)\to \Omega$ for the universal continuous derivation, so that $\Omega=K(\!(t)\!)dt$. We extend the valuation $v$ to differentials by setting $v(dt) = 1$. We then have $v(z)\leq v(dz)$ for all $z$. 

\begin{prop} 
\label{prop:curve-germ-case}
  Let $f, g \in K(\!(t)\!)$ with $v(f) < 0$  and $g \notin K$, and suppose that for every $n\geq 1$ 
  there exists an $h_n \in K(\!(t)\!)$ such that 
  \[ 
    f^n - g^n = h_n - h_n^p. 
  \]
  Then there exist integers $a,b\geq 0$ such that $f^{p^a}=g^{p^b}$.
\end{prop} 

\begin{proof}
Since $K$ is perfect, an element $z$ of $K(\!(t)\!)$ is a $p$-th power if and only if $dz = 0$. Repeatedly taking $p$-th roots of $f$ or $g$, we may therefore assume that $df$ and $dg$ are nonzero. Note that we keep the assumption of the lemma by taking $p$-th roots because if for example  $f = (f_1)^p$, then 
\[
  f_1^n - g^n = f_1^n - (f_1^n)^p + f^n - g^n = (f_1^n + h_n) - (f_1^n + h_n)^p.
\]
After these preparations, our goal is now to show that $f=g$.  Possibly exchanging $f$ and $g$, we may assume that $v(f)\leq v(g)$. Set $\varepsilon = g/f$, so that $v(\varepsilon)\geq 0$. We need to show that $\varepsilon =1$, so suppose otherwise.

The basic idea is to differentiate both sides of 
\[
  f^n(1-\varepsilon^n) = f^n - g^n = h_n - h_n^p
\]
and look at valuations, making use of the fact that $d (h_n^p)$ disappears. This forces the valuation of $d(f^n - g^n)$ to be much less negative than the valuation of $f^n-g^n$ forcing a number (that grows linearly with $n$) of certain  initial coefficients of $f^n$ and $g^n$ to agree. 

The actual argument below does not refer to coefficients.  Rather, we need to consider $n$ coprime to $p$ and such that $\varepsilon^n\neq 1$. All values of $n$ from now on will be assumed to satisfy these assumptions. Since we assumed $\varepsilon\neq 1$, the set of such $n$ is infinite, and it makes sense to talk about asymptotics as $n\gg 0$. We use the notation $O(1)$ to denote a bounded function in $n$.

\begin{lemma}
  Consider values of $n\geq 0$ which are coprime to $p$ and such that $\varepsilon^n \neq 1$. Then:
  \begin{enumerate}[(a)]
    \item $v(1-\varepsilon^n)=O(1)$,
    \item $v(f^n - g^n) = nv(f) + O(1)$,
    \item $v(f^n - g^n)< 0$ for $n\gg 0$,
    \item $v(f^n - g^n) = pv(h_n)$ for $n\gg 0$,
    \item $v(d(f^n - g^n)) \geq v(h_n) $.
  \end{enumerate}
\end{lemma}

\begin{proof}
Assertion (a) follows from 
\[
  1-\varepsilon^n = \prod_{\xi^n = 1} (1 -\xi \varepsilon).
\]
and $(1 -\xi \varepsilon)(0) = 1-\xi\varepsilon(0) \not= 0$ unless $\xi = \varepsilon(0)^{-1}$. So in fact $v(1-\varepsilon^n)$ takes at most two values. Part (b) follows since $f^n - g^n = f^n(1-\varepsilon^n)$, and part (c) holds since $v(f)<0$, so that $nv(f) + O(1) < 0$ for $n\gg 0$. 

For (d), note that (c) implies that $v(h_n-h_n^p)<0$ for $n\gg 0$. But then $v(h_n)$ is negative and $v(h_n - h_n^p) = pv(h_n)$ by the triangle inequality. For (e), we write 
\[ 
  d(f^n - g^n) = d(h_n - h_n^p) = d h_n, 
\]
and $v(dh) \geq v(h)$ holds for any $h$. 
\end{proof}

Now, we combine parts (d) and (e) of the lemma to obtain the inequality:
\begin{equation} \label{eqn:first-ineq}
  v\left( d(f^n - g^n)\right) \geq v(h_n) = \frac{1}{p}v(f^n - g^n)  = \frac{n}{p}v(f) + O(1).
\end{equation}
On the other hand, we have
\[ 
 d(f^n - g^n) = nf^{n-1} x_n, 
  \qquad 
  x_n := (1-\varepsilon^n)df - f\varepsilon^{n-1}d\varepsilon
\]
and
\begin{equation} \label{eqn:second-ineq}
  v\left(d(f^n - g^n)\right) = (n-1) v(f) + v(x_n) = nv(f) + v(x_n) + O(1).
\end{equation}
Combining \eqref{eqn:first-ineq} and \eqref{eqn:second-ineq} we obtain
\begin{equation} \label{eqn:third-ineq}
  v(x_n) \geq cn + O(1), \qquad c = -\left(1-\frac 1 p\right)v(f) > 0.
\end{equation}

In the rest of the proof, we shall estimate $v(x_n)$ from above for certain values of $n$, obtaining a contradiction with \eqref{eqn:third-ineq}, showing that $\varepsilon = 1$. We first note that 
\[ 
  x_n = \frac{ d(f^n - g^n)}{nf^{n-1}} = df - \varepsilon^{n-1} dg.
\]

\medskip

\noindent \emph{Case 1.} Suppose that $v(\ep)>0$. Then for $n \gg 0$ we have 
\[
  v(x_n) =  v(df  - \varepsilon^{n-1} dg) = v(df) = O(1),
\]
contradicting \eqref{eqn:third-ineq}. 

\medskip

\noindent \emph{Case 2.} Suppose that $\ep$ is a root of unity. Then $x_n$ takes only finitely many values and so $v(x_n) = O(1)$ is bounded, again contradicting \eqref{eqn:third-ineq}.

\medskip

\noindent \emph{Case 3.} Suppose that $v(\varepsilon) = 0$, but $\ep$ is not a root of unity. Take $n$ prime to $p$ and calculate
\[ 
  v(x_{n+p}-x_n) = v\left((\varepsilon^{n-1} - \varepsilon^{n+p-1})dg\right) = (n-1) v(\varepsilon) + v(1- \varepsilon^{p})+ v(dg) = O(1).
\]
On the other hand, \eqref{eqn:third-ineq} implies
\[  
    v(x_{n+p}-x_n)  \geq \min \{ v(x_{n+p}), v(x_n)\} \geq cn + O(1),
\]
and we obtain a contradiction again.
\end{proof}

\subsection{Concluding the proof} 
\label{ss:c_implies_d}

The goal of this section is to prove the implication \ref{thmitem:H1} $\Rightarrow$ \ref{thmitem:top} of Theorem~\ref{thm:main}, concluding its proof. 

\begin{prop} \label{prop:equal on cohomology}
Let $Y$ be an integral $\bF_p$-scheme of finite type over a field, and let $X$ be a quasi-affine $\bF_p$-scheme of finite type over a field.
  Let $f, g \colon Y \to X$ be a pair of maps whose image has positive dimension. Suppose that the two maps
  \[ 
    f^*, g^*\colon \rH^1(X, \bF_p)\to \rH^1(Y, \bF_p)
  \]
  are equal. Then $f$ and $g$ induce the same map $|f| = |g| : |Y| \to |X|$ on the underlying sets. 
\end{prop}

\begin{proof}
If $h \colon Y' \to Y$ is a surjective map of integral $\bF_p$-schemes respectively of finite type over a field, then we may replace $f$ and $g$ by $fh$ and $gh$. We apply this reduction to the normalization of a suitably chosen irreducible component of the base change of $Y$ to an algebraically closed field $K$.

Next, we choose an embedding $i: X \inj X'$ into an affine variety $X'$. Because the map $|i|$ is a monomorphism $|X| \inj |X'|$, we may replace $X$ by $X'$, compose $f$ and $g$ by $i$, and thus assume that $X$ is affine from the start.

The map $|Y| \to |X|$ associated to $f$ (resp.\ $g$) is determined by the restriction to closed points $Y(K) = |Y|_0$ of $|Y|$. Indeed, for any $y \in |Y|$ let $Z_y \subseteq Y$ denote the Zariski closure of $y$ in $Y$. Then $f(y)$ is the generic point of the Zariski closure of $f(|Z_y| \cap |Y|_0)$ in $|X|$.

Because $f(Y)$ and $g(Y)$ have positive dimension, both $f$ and $g$ are not constant when restricted to closed points $|Y|_0$.

For every closed point $y \in Y$ pick an affine irreducible curve $C_y \subseteq Y$ passing through $y$ and points $z_1,z_2$ depending on $y$ with $f(y) \not= f(z_1)$ and $g(y) \not= g(z_2)$, see \cite{mumford:AV}*{Lemma, p.~56}\footnote{Note that the proof of \cite{mumford:AV}*{Lemma, p.~56} only constructs a curve passing through two points. But the proof immediately generalizes to any finite set of closed points.}.

Thus $f$ and $g$ restricted to $C_y$ are nonconstant maps $f_y := f|_{C_y}$ and $g_y := g|_{C_y}$. If we assume \pref{prop:equal on cohomology} in the case of $\dim(Y)=1$, then the restrictions 
$|f_y| = |g_y|$ agree. As the union of these curves $C_y$ with $y$ ranging over all closed points of $Y$ covers $|Y|_0$, the two maps $|f|$ and $|g|$ agree on all of $|Y|$ and the proof is achieved. This reduces the proof to the case $\dim(Y) = 1$ with $Y$ affine. By replacing $Y$ by its normalization, we may assume that $Y$ is smooth because normal curves over $K$ are smooth, $K$ being algebraically closed.

Let $Y = \Spec(B)$ and $X = \Spec(A)$, and the maps $f$ and $g$ are given by homomorphisms $f,g : A \to B$. By assumption, neither $f$ nor $g$ is constant, hence their image is not contained in $K = \bigcap_{n\geq 0} B^{p^n}$. Let $n$ be maximal with $f(A) \subseteq B^{p^n}$, in other words we can factor $f = F_X^n \circ f'$ with $f'(A) \not\subseteq B^p$. Since $F_X$ acts as identity on $|X|$, we may replace $f$ by $f'$ and similarly for $g$ in order to reduce to the case where $f(A)$ and $g(A)$ are not contained in $B^p$.  The result then follows from the more precise proposition below.
\end{proof}

\begin{prop} \label{prop:equal on cohomology precise version}
Let $Y = \Spec(B)$ be an affine smooth connected curve over the algebraically closed field $K$ of characteristic $p > 0$. Let $X=\Spec(A)$ be an affine $\bF_p$-scheme of finite type over a field. 
Let $f, g \colon A \to B$ be a pair of maps such that their respective images as maps $f,g \colon Y \to X$  have positive dimension, and such that $f(A)$ and $g(A)$ are not contained in $B^p$. Suppose that the two maps
  \[ 
    f^*, g^*\colon \rH^1(X, \bF_p)\to \rH^1(Y, \bF_p)
  \]
  are equal. Then $f$ equals $g$.
\end{prop}

\begin{proof}
For each $t \in A$ we may compose $f$ and $g$ with the map $t\colon X \to \bA^1_{\bF_p}$. The assumptions of the proposition are preserved by $tf$ and $tg$ if 
\[
t \in S := A \setminus \Big(f^{-1}(B^p) \cup g^{-1}(B^p)\Big).
\]
Note that if for example $tf$ is constant, then  $\bF_p[t] \to B$, $t \mapsto f(t)$ factors over a subfield of $B$, hence over the maximal such $K \subseteq B$. In that case $t$ would be a $p$-th power in $B$. So asking  $t \in S$  also guarantees that $tf$ and $tg$ are both not constant.

If the proposition holds for all such compositions $tf,tg: Y \to \bA^1_{\bF_p}$, then $f(t) = g(t)$ for all $t \in S$.
Since $S$ is the complement in $A$ by two proper $\bF_p$-subvector spaces, the set $S$ generates $A$ as an $\bF_p$-algebra. This reduces the proof of the proposition to the case 
\[
X = \bA^1_{\bF_p} = \Spec\big(\bF_p[t]\big).
\]

By assumption all $\bF_p$-torsors $C \to \bA^1_{\bF_p}$ pull back with $f^*$ and $g^*$ to isomorphic torsors. We apply this to the Artin-Schreier torsor and also its pullback via $t \mapsto t^n$, i.e, the torsor $t^n = x - x^p$. That isomorphism of torsors is expressed by the existence of $h_n \in B$ for each $n \geq 1$ such that 
\begin{equation} \label{eq:torsors are equal}
  f^n - g^n = h_n - (h_n)^p.
\end{equation}

We now factor $f,g \colon Y \to \bA^1_{\bF_p}$ as $K$-linear maps $\ph, \gamma \colon Y \to \bA^1_K$ followed by the projection $\bA^1_K \to \bA^1_{\bF_p}$. Still $d \ph = \ph^* dt$ and $d \gamma = \gamma^* dt$ are non-zero and \eqref{eq:torsors are equal} holds for the elements $f= \ph$ and $g=\gamma$ in $B$. We may now compactifiy $Y$ to a smooth projective curve $\overline Y$ and extend the maps to $\overline{\ph}, \overline{\gamma}  : \overline Y \to \bP^1_K$. Since $\overline{\ph}$ is not constant, we may look at the local field $K(\!(T)\!)$ at a closed point $y \in Y$ where $\overline{\ph}$ has a pole. Here \pref{prop:curve-germ-case} applies and yields $a,b \geq 0$ with $f^{p^a} = g^{p^b}$. Since $df$ and $dg$ are both non-zero, we may cancel powers of $p$ on both sides until none remain and so $f = g$. 
\end{proof}

\bibliographystyle{alphaSGA}


\end{document}